\documentclass[11pt]{amsart}

\usepackage[english]{babel}
\usepackage[T1]{fontenc}

\usepackage{amsmath,amssymb,amsthm}
\usepackage{newtxtext}
\usepackage{newtxmath}
\usepackage{microtype}

\usepackage{hyperref}

\usepackage{url}

\newtheorem{theorem}{Theorem}[section]

\newtheorem{definition}[theorem]{Definition}
\newtheorem{remark}[theorem]{Remark}
\newtheorem{proposition}[theorem]{Proposition}
\newtheorem{lemma}[theorem]{Lemma}

\DeclareMathOperator{\Aut}{Aut}

\title{The Multiorbital Bivariate Chromatic Polynomial: A Subgroup-Lattice Refinement}

\author{Melanie Gerling}
\address{Independent Researcher, Germany}
\email{melanie.gerling.math@gmail.com}


\begin{document}

\maketitle

\begin{abstract}
We introduce the multiorbital bivariate chromatic polynomial
\[
F_{\Gamma}(G;x,y)
=
\sum_{H\leq G}
\frac{1}{|H|}
\sum_{h\in H}
P_{\Gamma/h}(x,y),
\]
which aggregates the orbital bivariate chromatic polynomials associated
with all subgroups of a finite group acting on a graph. The construction
provides a subgroup-lattice refinement of the orbital bivariate chromatic
polynomial and is motivated by a limitation of the latter's distinguishing
power. In particular, Dohmen and Lange-Geisler asked whether non-isomorphic
graphs can have identical orbital bivariate chromatic polynomials. We
answer this question affirmatively by exhibiting two non-isomorphic graphs
($P_3$ and $K_2\mathbin{\dot\cup}K_1$ under $C_2$-actions) with identical
orbital bivariate chromatic polynomials. The corresponding multiorbital
bivariate chromatic polynomials are different, showing that the new
invariant distinguishes graph actions that the orbital invariant does not.

We develop the structural theory of the multiorbital bivariate chromatic
polynomial. An equivalent element-wise representation is given by
\[
F_{\Gamma}(G;x,y)
=
\sum_{g\in G}c_G(g)P_{\Gamma/g}(x,y),
\qquad
c_G(g)
=
\sum_{\langle g\rangle\leq H\leq G}\frac{1}{|H|}.
\]
The coefficient function depends only on the cyclic subgroup generated by
the element and is constant on conjugacy classes. This yields
decompositions by cyclic subgroups and conjugacy classes and gives a
natural interpretation in terms of the incidence algebra of the subgroup
lattice. After normalization, the coefficients define a probability
distribution obtained by choosing a subgroup uniformly and then an element
uniformly from that subgroup. We also establish a diagonal multiplicativity
relation for disjoint unions and a weighted cycle-index expression for
edgeless graphs. The resulting refinement is not determined by the
orbital bivariate chromatic polynomial, whereas the converse question
remains open.
\end{abstract}


\section{Introduction}

Graph polynomials provide a powerful way of encoding structural information
about graphs. The classical chromatic polynomial, introduced in early work
of Birkhoff \cite{Birkhoff1912}, counts proper vertex colorings and has
become a central object in algebraic and enumerative graph theory. A
bivariate generalization was introduced by Dohmen, P\"onitz, and Tittmann
\cite{DohmenPoenitzTittmann2003}. Their polynomial simultaneously extends
several classical graph polynomials and distinguishes between proper and
improper colors. Further developments were given by Averbouch, Godlin,
and Makowsky \cite{AverbouchGodlinMakowsky2010}.

Symmetry leads naturally to a second layer of structure. If a finite group
$G$ acts by automorphisms on a graph $\Gamma$, one may ask not only for the
number of colorings of $\Gamma$, but also for colorings modulo the action of
$G$. This viewpoint gives rise to orbital graph polynomials. Cameron and
Kayibi \cite{CameronKayibi2007} studied orbital chromatic and flow
polynomials, while Cameron, Jackson, and Rudd developed a broader
orbit-counting polynomial framework encompassing orbital coloring, flow,
and tension-type enumerators \cite{CameronJacksonRudd2008}.

A bivariate version of this framework was introduced by Dohmen and
Lange-Geisler \cite{DohmenLangeGeisler2026}. Their orbital bivariate
chromatic polynomial combines the bivariate chromatic polynomial with a
group action and provides the starting point for the present construction.

An important question concerning the distinguishing power of this invariant
was raised by Dohmen and Lange-Geisler. In particular, they asked whether
there exist non-isomorphic graphs with the same orbital bivariate chromatic
polynomial. This question probes how much structural information is retained
by the orbital aggregation of the quotient graph polynomials. It also
naturally raises the question of whether a refinement of this aggregation
can retain information that is lost in the orbital polynomial.

The present paper answers this question by an explicit counterexample.
We exhibit two non-isomorphic graphs equipped with actions of the same
finite group whose orbital bivariate chromatic polynomials coincide.
The example is particularly small. Let
\[
\Gamma_1=P_3
\]
with vertex set $\{1,2,3\}$ and let
\[
G=C_2=\{e,g\},
\qquad
g=(1\,3),
\]
so that $g$ acts as the reflection of the path. On the other hand, let
\[
\Gamma_2=K_2\mathbin{\dot\cup} K_1,
\]
where $g=(1\,2)$ acts on the $K_2$ component and fixes the isolated
vertex. The graphs $\Gamma_1$ and $\Gamma_2$ are not isomorphic, since
$\Gamma_1$ is connected whereas $\Gamma_2$ is disconnected. Nevertheless,
their orbital bivariate chromatic polynomials are equal:
\[
OP_{\Gamma_1,C_2}(x,y)
=
OP_{\Gamma_2,C_2}(x,y)
=
\frac12x^3+\frac12x^2-xy.
\]

Thus, the orbital bivariate chromatic polynomial does not, in general,
distinguish non-isomorphic graphs. This provides an explicit example
answering the question raised by Dohmen and Lange-Geisler.

The same example also reveals a distinction between the orbital and
multiorbital constructions. Although the two graph actions above have the
same orbital bivariate chromatic polynomial, their multiorbital bivariate
chromatic polynomials are different:
\[
F_{\Gamma_1}(C_2;x,y)
=
\frac32x^3+\frac12x^2-3xy+y
\]
whereas
\[
F_{\Gamma_2}(C_2;x,y)
=
\frac32x^3+\frac12x^2-2xy.
\]
In particular,
\[
F_{\Gamma_1}(C_2;x,y)-F_{\Gamma_2}(C_2;x,y)
=
y(1-x)
\neq 0.
\]
Thus, the subgroup-lattice aggregation distinguishes the two graph actions
even though their orbital bivariate chromatic polynomials coincide.

This example shows that the multiorbital bivariate chromatic polynomial
contains information that is not determined by the orbital bivariate
chromatic polynomial. In particular, the multiorbital construction is not
merely a reparametrization of the orbital polynomial: the subgroup-lattice-
dependent coefficients can retain information that is lost when the
quotient graph polynomials are combined in the single orbital aggregation.
The example therefore provides a concrete instance in which the
subgroup-lattice refinement increases the distinguishing power of the
resulting invariant.

The main contributions of this paper are the introduction and structural
analysis of this subgroup-lattice aggregation and the investigation of its
distinguishing power. The coefficient function $c_G$ separates the
group-theoretic and graph-theoretic components of the construction. We show
that it depends only on the cyclic subgroup generated by an element and is
constant on conjugacy classes. We further identify it as a zeta transform
on the subgroup lattice and obtain the corresponding M\"obius inversion
formula. After normalization, the same coefficients admit a probabilistic
interpretation as the distribution obtained by choosing a subgroup
uniformly and then an element uniformly from that subgroup.

We also study basic structural properties of the resulting polynomial,
including its behavior under disjoint unions, its specialization to the
ordinary chromatic polynomial, and its restriction to edgeless graphs,
where a weighted cycle-index expression arises.

The purpose of the construction is not to provide a new orbit-counting
theorem. Rather, the subgroup lattice is used here to define a new
coefficient system for the quotient graph polynomials occurring in orbital
bivariate chromatic enumeration. This differs from classical
subgroup-lattice refinements of Burnside's lemma, such as the
M\"obius-theoretic orbit-size enumeration of Klass \cite{Klass1976}.

The construction can be viewed as the third step in the following sequence
of refinements:
\[
P_\Gamma (x,y)
\;\longrightarrow\;
OP_{\Gamma,G}(x,y)
\;\longrightarrow\;
F_\Gamma(G;x,y).
\]

The first step is provided by the bivariate chromatic polynomial of
Dohmen, P\"onitz, and Tittmann \cite{DohmenPoenitzTittmann2003} and its
subsequent developments. The second step incorporates a group action
through the orbital framework, in particular the construction of Dohmen
and Lange-Geisler \cite{DohmenLangeGeisler2026}. The third step is the
operation studied in the present paper.

Let $\Gamma$ be a finite graph and let $G\leq\Aut(\Gamma)$ be finite. For
an element $g\in G$, let $\Gamma/g$ denote the quotient graph obtained from
the orbits of the cyclic subgroup $\langle g\rangle$ on $V(\Gamma)$. We
define
\begin{equation}\label{eq:intro-multi-orbital}
F_{\Gamma}(G;x,y)
:=
\sum_{H\leq G}
\frac{1}{|H|}
\sum_{h\in H}
P_{\Gamma/h}(x,y).
\end{equation}

Interchanging the finite sums gives
\begin{equation}\label{eq:intro-elementwise}
F_{\Gamma}(G;x,y)
=
\sum_{g\in G}
c_G(g)P_{\Gamma/g}(x,y),
\qquad
c_G(g)
=
\sum_{\langle g\rangle\leq H\leq G}
\frac{1}{|H|}.
\end{equation}

Thus, the construction provides a common framework in which the subgroup
lattice, the cyclic and conjugacy structure of the acting group, and the
quotient graph polynomials appear in a single polynomial invariant.

Consequently, the multiorbital bivariate chromatic polynomial is not
determined by the orbital bivariate chromatic polynomial. Whether the
converse holds---that is, whether equality of multiorbital polynomials
implies equality of orbital polynomials---remains an open question.


\section{Related Work}

\subsection{Bivariate chromatic polynomials}\label{subsec:bivariate}

The bivariate chromatic polynomial used throughout this paper was
introduced by Dohmen, P\"onitz, and Tittmann
\cite{DohmenPoenitzTittmann2003}. Let $\Gamma=(V,E)$ be a finite,
undirected graph without loops or multiple edges, and let
\[X=Y\dot\cup Z\]
be a partition of a set of colors $X$, where $|X|=x$ and $|Y|=y$. A
coloring
\[\varphi:V\to X\]
is called \emph{generalized proper} if, for every edge $\{u,v\}\in E$,
\[\varphi(u),\varphi(v)\in Y
\quad\Longrightarrow\quad
\varphi(u)\neq\varphi(v).\]
Thus, adjacent vertices may receive the same color only if that color
belongs to $Z$. If $|X|=\lambda$ and $|Y|=\mu$, such a coloring is also
called a \emph{$\mu$-proper $\lambda$-coloring}, the parameters being
\emph{admissible} if $0\leq\mu\leq\lambda$.

The number of generalized proper colorings is a polynomial in $x$ and $y$,
called the \emph{bivariate chromatic polynomial} of $\Gamma$ and denoted
by
\[P_\Gamma(x,y).\]
It satisfies
\[P_\Gamma(x,0)=x^{|V|}
\qquad\text{and}\qquad
P_\Gamma(x,x)=\chi_\Gamma(x),\]
where $\chi_\Gamma$ denotes the ordinary chromatic polynomial. Moreover, the
bivariate chromatic polynomial is multiplicative with respect to disjoint
unions:
\[P_{\Gamma_1\sqcup\Gamma_2}(x,y)
=
P_{\Gamma_1}(x,y)P_{\Gamma_2}(x,y).\]

Averbouch, Godlin, and Makowsky \cite{AverbouchGodlinMakowsky2010}
subsequently studied extensions and further structural properties of the
bivariate chromatic polynomial. In the present paper, we apply the bivariate
chromatic polynomial to quotient graphs $\Gamma/g$ arising from group
actions.

\subsection{Orbital graph polynomials}

The interaction between graph polynomials and group actions has been
studied in several settings. Cameron and Kayibi \cite{CameronKayibi2007}
investigated orbital chromatic and flow polynomials, incorporating group
actions into the enumeration of colorings and flows. Cameron, Jackson,
and Rudd \cite{CameronJacksonRudd2008} developed a broader orbit-counting
polynomial framework based on dual pairs of matrices and group actions.

A bivariate version of the orbital framework was introduced by Dohmen and
Lange-Geisler \cite{DohmenLangeGeisler2026}. For a finite graph $\Gamma$
and a finite group $G\leq\Aut(\Gamma)$, their orbital bivariate chromatic
polynomial is defined as the number of equivalence classes of generalized
proper colorings of $\Gamma$ under the action of $G$; in the terminology of
Section~\ref{subsec:bivariate}, these are the $\mu$-proper
$\lambda$-colorings with $\lambda=x$ and $\mu=y$. By Burnside's lemma, it
admits the representation
\[OP_{\Gamma,G}(x,y)
=
\frac{1}{|G|}
\sum_{g\in G}
P_{\Gamma/g}(x,y).\]

Here, $\Gamma/g$ denotes the quotient graph obtained by identifying the
vertices in each cycle of the permutation $g$. The quotient may contain
loops and multiple edges. A loop occurs whenever an edge of $\Gamma$ has both endpoints in
the same cycle of $g$. In accordance with the convention used by
Dohmen and Lange-Geisler \cite{DohmenLangeGeisler2026}, loops are retained,
while parallel edges are replaced by single edges, since they do not
affect the bivariate chromatic polynomial. In the presence of a loop,
the incident vertex cannot receive a color from $Y$, since the generalized
properness condition would require this color to be different from
itself. Thus, a vertex incident with a loop may only receive a color from
$Z$.

The quotient polynomial $P_{\Gamma/g}$ is the fixed-point contribution of
the element $g$ in the Burnside average. Indeed, a coloring fixed by $g$
is constant on the $\langle g\rangle$-orbits of $V(\Gamma)$ and hence
corresponds to a coloring of the quotient graph $\Gamma/g$.

The present construction starts from these element-wise contributions but
modifies the aggregation procedure: instead of averaging over a fixed
acting group, we consider the orbital polynomial for every subgroup
$H\leq G$ and aggregate over the subgroup lattice.

\subsection{Burnside's lemma and subgroup-lattice methods}

Burnside's lemma expresses the number of orbits of a finite group action in
terms of fixed-point numbers of individual group elements. If a finite
group $G$ acts on a finite set $S$, then
\[|S/G|
=
\frac{1}{|G|}
\sum_{g\in G}
|\operatorname{Fix}(g)|.\]
This element-wise averaging is the basic mechanism underlying the orbital
bivariate chromatic polynomial.

The subgroup lattice provides a finer source of information about a group
action. In particular, Klass \cite{Klass1976} established a
M\"obius-theoretic generalization of Burnside's combinatorial lemma. Let a
finite group $G$ act on a finite set $S$, and for a subgroup $K\leq G$
define
\[F_K
:=
\left|
\left\{
s\in S:
\sigma s=s
\text{ for all }\sigma\in K
\right\}
\right|.\]
Thus, $F_K$ counts the points fixed by every element of $K$.

If $E_k$ denotes the number of $G$-orbits in $S$ having cardinality $k$,
Klass's formula expresses $E_k$ in terms of subgroup fixed-point data:
\begin{equation}\label{eq:klass-theorem}
E_k
=
\frac{1}{k}
\sum_{\substack{H\leq G\\ [G:H]=k}}
\sum_{H\leq K\leq G}
\mu(H,K)F_K,
\end{equation}
where $\mu$ denotes the M\"obius function of the subgroup lattice of $G$
\cite{Klass1976}.

Klass's theorem and the present construction use the subgroup lattice for
different purposes. In Klass's setting, the input consists of subgroup
fixed-point numbers $\{F_K\}_{K\leq G}$, and M\"obius inversion is used
to recover the distribution of orbit sizes. In the present setting, the
basic data are the quotient graph polynomials
\[\{P_{\Gamma/g}(x,y)\}_{g\in G}.\]
For each subgroup $H\leq G$, we first form
\[OP_{\Gamma,H}(x,y)
=
\frac{1}{|H|}
\sum_{h\in H}
P_{\Gamma/h}(x,y),\]
and then aggregate these subgroup-indexed polynomials:
\[F_{\Gamma}(G;x,y)
=
\sum_{H\leq G}
OP_{\Gamma,H}(x,y).\]

Thus, the subgroup lattice in the present construction is not used to
recover orbit-size information. Rather, it provides the indexing structure
for a family of averages of quotient graph polynomials.

There is nevertheless a direct structural connection with M\"obius
inversion. For $g\in G$, the coefficient
\[c_G(g)
=
\sum_{\substack{H\leq G\\g\in H}}
\frac{1}{|H|}\]
is the value of the upper zeta transform of the function
$H\mapsto |H|^{-1}$ at the cyclic subgroup $\langle g\rangle$. M\"obius
inversion therefore describes the weights occurring in the multiorbital
graph polynomial. This use of M\"obius inversion is structural rather than
an orbit-size enumeration.


\section{The multiorbital bivariate chromatic polynomial}

We now introduce the multiorbital bivariate chromatic polynomial and derive
its basic structural properties. Throughout this section, $\Gamma$ is a
finite graph and $G\leq\Aut(\Gamma)$ is a finite group acting on $\Gamma$.

We use the definitions and conventions for orbital graph polynomials and
quotient graphs adopted by Cameron and Kayibi
\cite{CameronKayibi2007} and, in the bivariate setting, by Dohmen and
Lange-Geisler \cite{DohmenLangeGeisler2026}.

For an element $g\in G$, the quotient graph $\Gamma/g$ is obtained from
the orbits of $\langle g\rangle$ on $V(\Gamma)$. Thus, the vertices of
$\Gamma/g$ correspond to the $\langle g\rangle$-orbits on
$V(\Gamma)$. An edge is present between two distinct orbit-vertices
whenever an edge of $\Gamma$ joins the corresponding orbits. If an edge
of $\Gamma$ has both endpoints in the same orbit, the corresponding
quotient vertex carries a loop.

The relevance of this quotient construction follows from Burnside's lemma.
A coloring of $\Gamma$ fixed by $g$ must be constant on every
$\langle g\rangle$-orbit. Hence the fixed colorings of $g$ are naturally
identified with colorings of the quotient graph $\Gamma/g$.

We use the following formulation of the orbital bivariate chromatic
polynomial of Dohmen and Lange-Geisler; see
\cite[Theorem~1]{DohmenLangeGeisler2026}.

\begin{proposition}[Orbital bivariate chromatic polynomial]
Let $\Gamma$ be a finite graph and let $G\leq\Aut(\Gamma)$. The orbital
bivariate chromatic polynomial satisfies
\[OP_{\Gamma,G}(x,y)
=
\frac{1}{|G|}
\sum_{g\in G}
P_{\Gamma/g}(x,y).\]
In particular, the polynomial $P_{\Gamma/g}$ is the fixed-point
contribution of the element $g$ in the corresponding Burnside average.
\end{proposition}

\subsection{Coefficient structure of the orbital polynomial}
\label{subsec:orbital-coefficient-structure}

The quotient representation of the orbital bivariate chromatic
polynomial also yields information about its highest coefficients. We
first recall the broken-circuit expansion of the bivariate chromatic
polynomial for simple graphs, and then record the modifications
required for quotient graphs, which may contain loops.

\begin{lemma}[Broken-circuit expansion]
\label{lem:broken-circuit-expansion}
Let $H$ be a finite simple graph with $m=|V(H)|$. Then
\[
P_H(x,y)
=
\sum_{k\geq 0}\sum_{l=0}^{k}
(-1)^k b_{kl}(H)x^{m-k-l}y^l,
\]
where $b_{kl}(H)$ denotes the number of broken-circuit sets of size $k$
having $l$ connected components. In particular,
\[
[x^m]P_H(x,y)=1,
\qquad
[x^{m-1}]P_H(x,y)=[x^{m-1}y]P_H(x,y)=0,
\]
and
\[
[x^{m-2}y]P_H(x,y)=-|E(H)|.
\]
\end{lemma}

\begin{proof}
The stated expansion is the broken-circuit expansion of the
bivariate chromatic polynomial
\cite{DohmenPoenitzTittmann2003}. Here $A\subseteq E(H)$ denotes
a broken-circuit-free edge set, $V(A)$ denotes the set of
vertices incident with edges in $A$, and $b(A)$ denotes the
number of connected components of the graph $(V(A),A)$, with
$b(\varnothing)=0$. Since every broken-circuit-free edge set
is a forest, we have
\[
|A|+b(A)=|V(A)|\leq m,
\]
so all exponents of $x$ in the expansion are nonnegative.

The leading coefficient is obtained from $k=l=0$, since the empty set
is the unique broken-circuit set with $k=0$.

A pure $x$-term below the leading term would require $l=0$ and
$k\geq 1$, but every nonempty edge set has at least one connected
component. Hence $b_{k0}(H)=0$ for $k\geq 1$, which gives
$[x^{m-1}]P_H(x,y)=0$.

The monomial $x^{m-1}y$ would require $k+l=1$ with $l=1$, hence
$k=0$; but the empty edge set spans no connected component, so
$b_{01}(H)=0$. This gives $[x^{m-1}y]P_H(x,y)=0$.

Finally, the coefficient of $x^{m-2}y$ corresponds to $k=1$ and
$l=1$. Every broken circuit of a simple graph has size at least two,
because the shortest cycle has length three. Hence the
broken-circuit sets of size one are precisely the single edges, each
spanning one connected component, so
\[
b_{11}(H)=|E(H)|.
\]
Thus
\[
[x^{m-2}y]P_H(x,y)=-|E(H)|.
\]
\end{proof}

The quotient graphs $\Gamma/g$ need not be simple: an edge of $\Gamma$
whose endpoints lie in the same $\langle g\rangle$-orbit becomes a loop
in the quotient. The following lemma describes the effect of loops.

\begin{lemma}[Looped graphs]
\label{lem:loops}
Let $\Lambda$ be a finite graph, possibly with loops, let
$L=L(\Lambda)$ be its set of looped vertices, let $r=|L|$, and let
\[
\Lambda':=\Lambda\bigl[V(\Lambda)\setminus L\bigr]
\]
be the subgraph induced by the remaining vertices, with parallel
edges identified as usual, so that $\Lambda'$ is simple. Then
\[
P_{\Lambda}(x,y)=(x-y)^r\,P_{\Lambda'}(x,y).
\]
In particular, $x^{|V(\Lambda)|}$ is the unique monomial of
$P_{\Lambda}(x,y)$ with $x$-degree $|V(\Lambda)|$, and its
coefficient is $1$; every other monomial carries a factor $y$ and has
$x$-degree at most $|V(\Lambda)|-1$.
\end{lemma}

\begin{proof}
A vertex of $L$ can only receive a color from $Z$, since the loop
would otherwise require this color to differ from itself; this gives
$x-y$ independent choices for each such vertex. The
generalized-properness condition on an edge applies only if both
endpoint colors lie in $Y$. Edges of $\Lambda$ between $L$ and its
complement therefore impose no condition, and edges within the
complement impose the same condition as in $\Lambda'$. Hence the
generalized proper colorings of $\Lambda$ correspond bijectively to
the pairs consisting of a $Z$-coloring of $L$ and a generalized
proper coloring of $\Lambda'$, which proves the product formula.

For the final claim, expand $(x-y)^r$ by the binomial theorem. A pure
$x$-monomial of the product can arise only from the leading term
$x^r$ of $(x-y)^r$ combined with a pure $x$-monomial of
$P_{\Lambda'}$, and by Lemma~\ref{lem:broken-circuit-expansion} the
only such monomial is $x^{|V(\Lambda')|}$, with coefficient $1$.
Hence $x^{r+|V(\Lambda')|}=x^{|V(\Lambda)|}$ is the unique pure
$x$-monomial of $P_{\Lambda}$. Moreover, every monomial of the
product has total degree at most $|V(\Lambda)|$, so a monomial
carrying a factor $y$ cannot have $x$-degree $|V(\Lambda)|$; it has
$x$-degree at most $|V(\Lambda)|-1$.
\end{proof}

\begin{proposition}[Highest coefficients of the orbital polynomial]
\label{prop:orbital-high-coefficients}
Let $\Gamma$ be a finite simple graph with $n$ vertices and let
$G\leq\Aut(\Gamma)$ be finite. Then
\begin{align}
[x^n]OP_{\Gamma,G}(x,y)
&=
\frac{1}{|G|},
\label{eq:orbital-leading-coeff}\\
[x^{n-1}]OP_{\Gamma,G}(x,y)
&=
\frac{1}{|G|}
\left|
\left\{
g\in G:|V(\Gamma/g)|=n-1
\right\}
\right|,
\label{eq:orbital-second-coeff}\\
[x^{n-2}y]OP_{\Gamma,G}(x,y)
&=
-\frac{|E(\Gamma)|+t_E(\Gamma,G)}{|G|},
\label{eq:orbital-edge-coeff}
\end{align}
where $t_E(\Gamma,G)$ denotes the number of elements $g\in G$ that
exchange the endpoints of an edge of $\Gamma$ while fixing every
other vertex. In particular, the coefficient of $x^{n-1}$ records the
number of elements of $G$ whose action on $V(\Gamma)$ has exactly
$n-1$ orbits.
\end{proposition}

\begin{proof}
By the quotient representation of the orbital polynomial,
\[
OP_{\Gamma,G}(x,y)
=
\frac{1}{|G|}
\sum_{g\in G}P_{\Gamma/g}(x,y).
\]
Since $G\leq\Aut(\Gamma)$ acts faithfully on $V(\Gamma)$, every
non-identity element $g\in G$ moves at least one vertex, and the
orbit of that vertex has size at least two. Hence
\[
|V(\Gamma/g)|\leq n-1
\qquad
\text{for every }g\neq e.
\]

We first consider the coefficient of $x^n$. By
Lemma~\ref{lem:loops}, the $x$-degree of $P_{\Gamma/g}(x,y)$ equals
$|V(\Gamma/g)|$, which is at most $n-1$ for $g\neq e$. Hence only the
identity element contributes. Since $\Gamma$ is simple,
Lemma~\ref{lem:broken-circuit-expansion} gives
$[x^n]P_\Gamma(x,y)=1$. Therefore
\[
[x^n]OP_{\Gamma,G}(x,y)
=
\frac{1}{|G|}.
\]

Next consider the coefficient of $x^{n-1}$. By
Lemma~\ref{lem:broken-circuit-expansion}, the polynomial $P_\Gamma$
has no pure $x$-monomial below its leading term, so the identity does
not contribute. Let $g\neq e$. By Lemma~\ref{lem:loops}, the unique
monomial of $P_{\Gamma/g}(x,y)$ with $x$-degree $|V(\Gamma/g)|$ is
its leading term, of coefficient $1$, and every further monomial has
strictly smaller $x$-degree. Hence $g$ contributes to $[x^{n-1}]$
precisely when $|V(\Gamma/g)|=n-1$, and its contribution is $1$.
Therefore
\[
[x^{n-1}]OP_{\Gamma,G}(x,y)
=
\frac{1}{|G|}
\left|
\left\{
g\in G:|V(\Gamma/g)|=n-1
\right\}
\right|.
\]

It remains to determine the coefficient of $x^{n-2}y$. The identity
element contributes
\[
[x^{n-2}y]P_\Gamma(x,y)
=
-|E(\Gamma)|
\]
by Lemma~\ref{lem:broken-circuit-expansion}. Now let $g\neq e$ and put
\[
m:=|V(\Gamma/g)|.
\]

If $m\leq n-2$, then by Lemma~\ref{lem:loops} every monomial of
$P_{\Gamma/g}(x,y)$ carrying a factor $y$ has $x$-degree at most
$m-1\leq n-3$. A contribution to $x^{n-2}y$ would require a factor
$y$ together with $x$-degree $n-2$, which is impossible. Hence such
elements contribute nothing.

It remains to consider the case $m=n-1$. Then exactly one vertex is
lost in passing from $V(\Gamma)$ to the set of $\langle
g\rangle$-orbits, so $g$ has a single orbit $\{u,v\}$ of size $2$
while fixing every other vertex.

Suppose first that $\{u,v\}\notin E(\Gamma)$. Then no $\langle
g\rangle$-orbit contains an edge of $\Gamma$, so the quotient
$\Gamma/g$ is loopless and simple with $n-1$ vertices. Applying
Lemma~\ref{lem:broken-circuit-expansion} to $\Gamma/g$ yields
\[
[x^{(n-1)-1}y]\,P_{\Gamma/g}(x,y)
=
[x^{n-2}y]\,P_{\Gamma/g}(x,y)
=
0,
\]
so there is no contribution.

Suppose now that $\{u,v\}\in E(\Gamma)$. Then the edge $\{u,v\}$
becomes a loop at the vertex of $\Gamma/g$ corresponding to the orbit
$\{u,v\}$, while every other quotient vertex is unlooped. By
Lemma~\ref{lem:loops},
\[
P_{\Gamma/g}(x,y)=(x-y)\,P_{\Lambda}(x,y),
\]
where $\Lambda:=(\Gamma/g)'$ is the simple graph obtained from
$\Gamma/g$ by deleting the looped vertex; it has $n-2$ vertices. By
Lemma~\ref{lem:broken-circuit-expansion},
\[
[x^{n-2}]\,P_{\Lambda}(x,y)=1
\qquad\text{and}\qquad
[x^{(n-2)-1}y]\,P_{\Lambda}(x,y)
=
[x^{n-3}y]\,P_{\Lambda}(x,y)
=
0.
\]
Consequently,
\[
[x^{n-2}y]\,P_{\Gamma/g}(x,y)
=
[x^{n-3}y]\,P_{\Lambda}(x,y)
-
[x^{n-2}]\,P_{\Lambda}(x,y)
=
-1.
\]

Thus a non-identity element contributes to $[x^{n-2}y]$ precisely
when it exchanges the endpoints of an edge while fixing every other
vertex, and each such element contributes $-1$. By definition, the
number of these elements is $t_E(\Gamma,G)$. Therefore
\[
[x^{n-2}y]OP_{\Gamma,G}(x,y)
=
\frac{1}{|G|}
\bigl(
-|E(\Gamma)|-t_E(\Gamma,G)
\bigr),
\]
which completes the proof.
\end{proof}

\begin{remark}
The first identity immediately determines the order of the acting
group:
\[
|G|
=
\frac{1}{[x^n]OP_{\Gamma,G}(x,y)}.
\]
Combining the first and third identities gives
\[
|E(\Gamma)|+t_E(\Gamma,G)
=
-\frac{[x^{n-2}y]OP_{\Gamma,G}(x,y)}
{[x^n]OP_{\Gamma,G}(x,y)}.
\]
Together with the $x$-degree of $OP_{\Gamma,G}$, which equals $n$, the
orbital bivariate chromatic polynomial therefore determines
\[
|V(\Gamma)|,
\qquad
|G|,
\qquad
|E(\Gamma)|+t_E(\Gamma,G).
\]
The coefficient of $x^{n-2}y$ does not, in general, determine
$|E(\Gamma)|$ alone: it also records the number of elements of $G$
that exchange the endpoints of an edge while fixing all other
vertices. See Section~\ref{subsec:distinguishing-power} for a pair of
graph actions with equal orbital polynomials but different numbers of
edges.
\end{remark}

\begin{remark}
If
\[
t_E(\Gamma,G)=0,
\]
then the third identity simplifies to
\[
[x^{n-2}y]OP_{\Gamma,G}(x,y)
=
-\frac{|E(\Gamma)|}{|G|}.
\]
Thus the edge-count formula remains valid whenever no element of $G$
exchanges the endpoints of an edge. This hypothesis is satisfied, for
example, if $\Gamma$ is bipartite and every element of $G$ preserves
the bipartition, and more generally whenever no involution of $G$
interchanges adjacent vertices.
\end{remark}

\begin{remark}
The coefficient of $x^{n-1}$ admits a direct group-theoretic
interpretation. Since
\[
|V(\Gamma/g)|
=
\left|
\operatorname{Orb}_{\langle g\rangle}(V(\Gamma))
\right|,
\]
we have
\[
[x^{n-1}]OP_{\Gamma,G}(x,y)
=
\frac{1}{|G|}
\left|
\left\{
g\in G:
\left|
\operatorname{Orb}_{\langle g\rangle}(V(\Gamma))
\right|
= n-1
\right\}
\right|.
\]
For the faithful permutation action of $G$ on $V(\Gamma)$, these are
precisely the elements whose permutation consists of one transposition
and $n-2$ fixed points.
\end{remark}

The preceding proposition is the starting point for the present
construction. We retain the element-wise quotient polynomials but
change the way in which they are aggregated.

\begin{definition}[Multiorbital bivariate chromatic polynomial]
Let $\Gamma$ be a finite graph and let $G\leq\Aut(\Gamma)$ be finite.

For $g\in G$, let $\Gamma/g$ denote the quotient graph whose vertices are
the $\langle g\rangle$-orbits on $V(\Gamma)$. If two vertices of $\Gamma$
belong to the same $\langle g\rangle$-orbit and are adjacent in $\Gamma$,
the corresponding quotient vertex carries a loop. Distinct quotient
vertices
are adjacent whenever at least one edge of $\Gamma$ joins the
corresponding orbits.

For every subgroup $H\leq G$, define
\begin{equation}\label{eq:orbital-polynomial}
OP_{\Gamma,H}(x,y)
:=
\frac{1}{|H|}
\sum_{h\in H}
P_{\Gamma/h}(x,y).
\end{equation}
This is the orbital bivariate chromatic polynomial introduced by Dohmen
and Lange-Geisler \cite[Theorem~1]{DohmenLangeGeisler2026}, written here in
the variables $(x,y)$.

We define the \emph{multiorbital bivariate chromatic polynomial} of
$\Gamma$ relative to $G$ by
\begin{equation}\label{eq:multi-orbital}
F_{\Gamma}(G;x,y)
:=
\sum_{H\leq G}
OP_{\Gamma,H}(x,y).
\end{equation}
Equivalently,
\begin{equation}\label{eq:multi-orbital-expanded}
F_{\Gamma}(G;x,y)
=
\sum_{H\leq G}
\frac{1}{|H|}
\sum_{h\in H}
P_{\Gamma/h}(x,y).
\end{equation}
\end{definition}

The construction therefore consists of two levels of averaging. For each
fixed subgroup $H$, the polynomial $OP_{\Gamma,H}$ is obtained by averaging
the quotient contributions over the elements of $H$. The multiorbital
polynomial then aggregates these orbital polynomials over the complete
subgroup lattice of $G$. The term \emph{multiorbital} indicates that the
orbital polynomials of all subgroups $H\leq G$ enter the aggregation,
rather than only that of the acting group $G$ itself.

\begin{proposition}[Element-wise expansion]
For every finite group $G$ acting on $\Gamma$,
\begin{equation}\label{eq:coefficient-expansion}
F_{\Gamma}(G;x,y)
=
\sum_{g\in G}
c_G(g)\,
P_{\Gamma/g}(x,y),
\end{equation}
where
\begin{equation}\label{eq:coefficient}
c_G(g)
:=
\sum_{\substack{H\leq G\\g\in H}}
\frac{1}{|H|}.
\end{equation}
\end{proposition}

\begin{proof}
Starting from \eqref{eq:multi-orbital-expanded}, we interchange the two
finite sums:
\[
\begin{aligned}
F_{\Gamma}(G;x,y)
&=
\sum_{H\leq G}
\frac{1}{|H|}
\sum_{h\in H}
P_{\Gamma/h}(x,y)\\
&=
\sum_{g\in G}
\left(
\sum_{\substack{H\leq G\\g\in H}}
\frac{1}{|H|}
\right)
P_{\Gamma/g}(x,y).
\end{aligned}
\]
The expression in parentheses is precisely $c_G(g)$.
\end{proof}

The element-wise representation is the central structural decomposition of
the multiorbital polynomial. It separates the group-theoretic contribution,
encoded by $c_G$, from the graph-theoretic contribution, encoded by the
quotient polynomials $P_{\Gamma/g}$.

\subsection{Coefficient structure and cyclic subgroups}

The condition $g\in H$ can be expressed in terms of the cyclic subgroup
generated by $g$:
\[g\in H
\quad\Longleftrightarrow\quad
\langle g\rangle\leq H.\]
Consequently,
\begin{equation}\label{eq:cyclic-coefficient}
c_G(g)
=
\sum_{\langle g\rangle\leq H\leq G}
\frac{1}{|H|}.
\end{equation}

\begin{proposition}
The coefficient $c_G(g)$ depends only on the cyclic subgroup
$\langle g\rangle$. In particular, if
\[\langle g\rangle=\langle k\rangle,\]
then
\[c_G(g)=c_G(k).\]
\end{proposition}

\begin{proof}
Equation \eqref{eq:cyclic-coefficient} shows that $c_G(g)$ is determined
entirely by the collection of subgroups containing $\langle g\rangle$.
Hence it depends only on $\langle g\rangle$.
\end{proof}

In the following, $\varphi$ denotes Euler's totient function, so that
$\varphi(n)$ is the number of generators of a cyclic group of order $n$.

Let
\[\mathcal{C}(G)
:=
\{C\leq G:C\text{ is cyclic}\}\]
denote the set of cyclic subgroups of $G$. For $C\in\mathcal{C}(G)$, define
\[c_G(C)
:=
\sum_{C\leq H\leq G}\frac{1}{|H|}.\]
If $C=\langle g\rangle$, then $c_G(C)=c_G(g)$.

If $g$ and $h$ are generators of the same cyclic subgroup $C$, then
\[\langle g\rangle=\langle h\rangle=C.\]
Consequently, $g$ and $h$ induce the same orbit partition on
$V(\Gamma)$, and hence the corresponding quotient graphs are isomorphic.
Since $C$ has exactly $\varphi(|C|)$ generators, we obtain the following
decomposition.

For a cyclic subgroup $C\leq G$, let $\Gamma/C$ denote the quotient
graph obtained by identifying the vertices in each orbit of $C$ on
$V(\Gamma)$. If $C=\langle g\rangle$, then the $C$-orbits are precisely
the cycles of the permutation $g$, and hence $\Gamma/C=\Gamma/g$.
As above, loops are retained and parallel edges are replaced by
single edges.

\begin{proposition}[Cyclic-subgroup decomposition]
With the notation above,
\begin{equation}\label{eq:cyclic-subgroup-expansion}
F_{\Gamma}(G;x,y)
=
\sum_{C\in\mathcal{C}(G)}
\varphi(|C|)\,
c_G(C)\,
P_{\Gamma/C}(x,y),
\end{equation}
where, for any generator $g$ of $C$,
\[P_{\Gamma/C}(x,y)
:=
P_{\Gamma/g}(x,y).\]
\end{proposition}

\begin{proof}
Group the elements $g\in G$ in
\eqref{eq:coefficient-expansion} according to the cyclic subgroup
$C=\langle g\rangle$. For a fixed cyclic subgroup $C$, all generators of
$C$ have the same coefficient $c_G(C)$ and induce the same orbit
partition. Thus they give isomorphic quotient graphs and hence the same
bivariate chromatic polynomial. Since $C$ has $\varphi(|C|)$ generators,
its total contribution is
\[\varphi(|C|)\,
c_G(C)\,
P_{\Gamma/C}(x,y).\]
Summing over all cyclic subgroups gives
\eqref{eq:cyclic-subgroup-expansion}.
\end{proof}

The coefficient of the identity element has a particularly simple form:
\[c_G(1)
=
\sum_{H\leq G}\frac{1}{|H|}.\]
Since every subgroup contains the identity, while only some subgroups
contain a given non-identity element, we have
\[c_G(g)\leq c_G(1)
\qquad
\text{for every }g\in G.\]

\subsection{Conjugacy and quotient graphs}

The coefficient function is compatible with conjugation.

\begin{proposition}
Let $\Gamma$ be a finite graph and let $G\leq\Aut(\Gamma)$. Then
$c_G$ is constant on the conjugacy classes of $G$. Moreover, if
$g,h\in G$ are conjugate, then
\[\Gamma/g\cong\Gamma/h\]
and consequently
\[P_{\Gamma/g}(x,y)
=
P_{\Gamma/h}(x,y).\]
\end{proposition}

\begin{proof}
Suppose that
\[h=kgk^{-1}\]
for some $k\in G$. Conjugation by $k$ gives a bijection
\[H\longmapsto kHk^{-1}\]
on the subgroup lattice of $G$ and preserves subgroup orders. Moreover,
\[g\in H
\quad\Longleftrightarrow\quad
h\in kHk^{-1}.\]
Therefore
\[
\begin{aligned}
c_G(h)
&=
\sum_{\substack{K\leq G\\h\in K}}
\frac{1}{|K|}\\
&=
\sum_{\substack{H\leq G\\g\in H}}
\frac{1}{|kHk^{-1}|}\\
&=
c_G(g).
\end{aligned}
\]

Since $h=kgk^{-1}$, the automorphism $k$ maps
$\langle g\rangle$-orbits on $V(\Gamma)$ onto
$\langle h\rangle$-orbits. It therefore induces a graph isomorphism
\[\Gamma/g\cong\Gamma/h.\]
The equality of the corresponding bivariate chromatic polynomials follows.
\end{proof}

\begin{proposition}[Conjugacy-class decomposition]
Let $[g]$ denote the conjugacy class of $g\in G$. Then
\begin{equation}\label{eq:conjugacy-class-expansion}
F_{\Gamma}(G;x,y)
=
\sum_{[g]\subseteq G}
|[g]|\,c_G(g)\,P_{\Gamma/g}(x,y).
\end{equation}
\end{proposition}

\begin{proof}
By the preceding proposition, both $c_G(g)$ and
$P_{\Gamma/g}(x,y)$ are constant on conjugacy classes. Grouping the terms
in \eqref{eq:coefficient-expansion} by conjugacy class therefore gives
\eqref{eq:conjugacy-class-expansion}.
\end{proof}

In particular, if $G=S_n$ with its natural permutation action, then the
conjugacy classes are indexed by cycle type. The coefficient function
$c_G(g)$ consequently depends only on the cycle type of $g$. The quotient
polynomial, however, also depends on the particular action of $G$ on
$\Gamma$.

\subsection{Distinguishing power: an answer to an open question}
\label{subsec:distinguishing-power}

Dohmen and Lange-Geisler posed the question whether the orbital
bivariate chromatic polynomial can fail to distinguish two graph
actions that are distinguished by the corresponding multiorbital
bivariate chromatic polynomial. We answer this question affirmatively by
the following explicit example.

Let
\[
\Gamma_1=P_3
\]
with vertex set $\{1,2,3\}$, and let
\[
G=C_2=\{e,g\}, \qquad g=(1\,3).
\]
Thus $g$ acts as the reflection of the path.

Let
\[
\Gamma_2=K_2\mathbin{\dot\cup}K_1,
\]
where the vertices of the $K_2$ component are $1,2$, and let the same
group $G=C_2$ act by
\[
g=(1\,2),
\]
fixing the isolated vertex.

The graphs $\Gamma_1$ and $\Gamma_2$ are not isomorphic, since
$\Gamma_1$ is connected whereas $\Gamma_2$ is disconnected.

For $\Gamma_1=P_3$, we have
\[
P_{\Gamma_1/e}(x,y)
= P_{P_3}(x,y)
= (x-y)x^2+y(x-1)^2
= x^3-2xy+y.
\]

Under the nontrivial element $g=(1,3)$, the vertices $1$ and $3$ are
identified. Since they are not adjacent, the quotient graph is $K_2$,
and hence
\[
P_{\Gamma_1/g}(x,y)=x^2-y.
\]

Therefore
\[
\begin{aligned}
OP_{\Gamma_1,C_2}(x,y)
&= \frac12\left(
P_{\Gamma_1/e}(x,y)+P_{\Gamma_1/g}(x,y)
\right)\\
&= \frac12\left(
x^3-2xy+y+x^2-y
\right)\\
&= \frac12x^3+\frac12x^2-xy.
\end{aligned}
\]

For $\Gamma_2=K_2\mathbin{\dot\cup}K_1$, the identity element gives
\[
P_{\Gamma_2/e}(x,y)
= P_{K_2}(x,y)P_{K_1}(x,y)
= (x^2-y)x
= x^3-xy.
\]

For the nontrivial element $g=(1,2)$, the two adjacent vertices of the
$K_2$ component are identified. The quotient therefore consists of one
looped vertex together with one isolated vertex. By the loop formula of
Lemma~\ref{lem:loops},
\[
P_{\Gamma_2/g}(x,y)
= (x-y)x
= x^2-xy.
\]

Consequently,
\[
\begin{aligned}
OP_{\Gamma_2,C_2}(x,y)
&= \frac12\left(
x^3-xy+x^2-xy
\right)\\
&= \frac12x^3+\frac12x^2-xy.
\end{aligned}
\]

Thus
\begin{equation}
\label{eq:equal-orbital-example}
OP_{\Gamma_1,C_2}(x,y)
=
OP_{\Gamma_2,C_2}(x,y).
\end{equation}

The multiorbital polynomials, however, are different. Since the subgroup
lattice of $C_2$ consists of $\{e\}$ and $C_2$, the coefficients in the
element-wise representation are
\[
c_{C_2}(e)=1+\frac12=\frac32,
\qquad
c_{C_2}(g)=\frac12.
\]

Hence
\[
\begin{aligned}
F_{\Gamma_1}(C_2;x,y)
&= \frac32P_{\Gamma_1/e}(x,y)
+\frac12P_{\Gamma_1/g}(x,y)\\
&= \frac32(x^3-2xy+y)
+\frac12(x^2-y)\\
&= \frac32x^3+\frac12x^2-3xy+y.
\end{aligned}
\]

On the other hand,
\[
\begin{aligned}
F_{\Gamma_2}(C_2;x,y)
&= \frac32P_{\Gamma_2/e}(x,y)
+\frac12P_{\Gamma_2/g}(x,y)\\
&= \frac32(x^3-xy)
+\frac12(x^2-xy)\\
&= \frac32x^3+\frac12x^2-2xy.
\end{aligned}
\]

Therefore
\begin{equation}
\label{eq:different-multiorbital-example}
F_{\Gamma_1}(C_2;x,y)
-
F_{\Gamma_2}(C_2;x,y)
=
y(1-x)
\neq 0.
\end{equation}

Combining \eqref{eq:equal-orbital-example} and
\eqref{eq:different-multiorbital-example}, we obtain
\[
OP_{\Gamma_1,C_2}(x,y)
=
OP_{\Gamma_2,C_2}(x,y),
\qquad
F_{\Gamma_1}(C_2;x,y)
\neq
F_{\Gamma_2}(C_2;x,y).
\]

Thus the question raised by Dohmen and Lange-Geisler has an affirmative
answer. In particular, the orbital bivariate chromatic polynomial does
not distinguish non-isomorphic graphs in general. Moreover, the
example shows that the multiorbital bivariate chromatic polynomial is not
determined by the orbital bivariate chromatic polynomial: the two graph
actions have the same orbital bivariate chromatic polynomial but different
multiorbital bivariate chromatic polynomials.

\subsection{Disjoint unions and diagonal multiplicativity}

We next consider disjoint unions. Let
\[\Gamma=\Gamma_1\sqcup\Gamma_2\]
and suppose that the action of $G$ preserves both components. Thus, every
$\langle g\rangle$-orbit is contained entirely in either
$V(\Gamma_1)$ or $V(\Gamma_2)$. Consequently,
\[(\Gamma_1\sqcup\Gamma_2)/g
=
(\Gamma_1/g)\sqcup(\Gamma_2/g).\]

Since the bivariate chromatic polynomial is multiplicative with respect to
disjoint unions,
\[P_{(\Gamma_1\sqcup\Gamma_2)/g}(x,y)
=
P_{\Gamma_1/g}(x,y)
P_{\Gamma_2/g}(x,y).\]
This gives the following relation.

\begin{proposition}[Diagonal multiplicativity]
Let
\[\Gamma=\Gamma_1\sqcup\Gamma_2\]
and suppose that $G$ preserves both components. Then
\begin{equation}\label{eq:diagonal-multiplicativity}
F_{\Gamma_1\sqcup\Gamma_2}(G;x,y)
=
\sum_{g\in G}
c_G(g)\,
P_{\Gamma_1/g}(x,y)\,
P_{\Gamma_2/g}(x,y).
\end{equation}
\end{proposition}

\begin{proof}
For every $g\in G$,
\[(\Gamma_1\sqcup\Gamma_2)/g
=
(\Gamma_1/g)\sqcup(\Gamma_2/g).\]
Hence
\[P_{(\Gamma_1\sqcup\Gamma_2)/g}(x,y)
=
P_{\Gamma_1/g}(x,y)
P_{\Gamma_2/g}(x,y).\]
Substituting this identity into
\eqref{eq:coefficient-expansion} gives
\eqref{eq:diagonal-multiplicativity}.
\end{proof}

The relation is diagonal because the same group element $g$ acts on both
components. It should therefore not be confused with ordinary
multiplicativity. Indeed,
\[F_{\Gamma_1}(G;x,y)F_{\Gamma_2}(G;x,y)
=
\sum_{g,h\in G}
c_G(g)c_G(h)\,
P_{\Gamma_1/g}(x,y)
P_{\Gamma_2/h}(x,y),\]
where the two group elements are chosen independently. Thus, in general,
\[F_{\Gamma_1\sqcup\Gamma_2}(G;x,y)
\neq
F_{\Gamma_1}(G;x,y)F_{\Gamma_2}(G;x,y).\]

\subsection{M\"obius inversion on the subgroup lattice}
\label{subsec:moebius-inversion}

The coefficient function has a natural interpretation in terms of the
subgroup lattice of $G$. This provides a convenient M\"obius-theoretic
description of the weights occurring in the multiorbital polynomial.

For every subgroup $K\leq G$, define
\[a(K):=\frac{1}{|K|}\]
and
\[b_G(K)
:=
\sum_{K\leq H\leq G}a(H)
=
\sum_{K\leq H\leq G}\frac{1}{|H|}.\]
If $K=\langle g\rangle$, then
\[c_G(g)=b_G(K).\]

Let
$\mu_{\mathcal{H}(G)}$ denote the M\"obius function of the finite subgroup
lattice $\mathcal{H}(G)$.

\begin{proposition}[M\"obius inversion]
For every subgroup $K\leq G$,
\begin{equation}\label{eq:moebius-inversion}
\frac{1}{|K|}
=
\sum_{K\leq H\leq G}
\mu_{\mathcal{H}(G)}(K,H)\,
b_G(H).
\end{equation}
\end{proposition}

\begin{proof}
By definition,
\[b_G(K)
=
\sum_{K\leq H\leq G}a(H).\]
Thus $b_G$ is the upper zeta transform of the function $a$ on the
subgroup lattice. M\"obius inversion on the finite poset
$\mathcal{H}(G)$ gives
\[a(K)
=
\sum_{K\leq H\leq G}
\mu_{\mathcal{H}(G)}(K,H)b_G(H),\]
which is precisely \eqref{eq:moebius-inversion}.
\end{proof}

Thus
\[c_G(g)
=
\sum_{\langle g\rangle\leq H\leq G}\frac{1}{|H|}\]
is the value of the upper zeta transform at the cyclic subgroup
$\langle g\rangle$. The inversion is naturally formulated on subgroups
rather than on individual group elements, reflecting the fact that
$c_G(g)$ depends only on $\langle g\rangle$.

\subsubsection{Example: The cyclic group $C_p$}

Let $G=C_p$ be cyclic of prime order $p$. Its subgroup lattice consists
of the trivial subgroup
\[1:=\{e\}\]
and $C_p$. The M\"obius function is
\[\mu(1,1)=1,
\qquad
\mu(1,C_p)=-1,
\qquad
\mu(C_p,C_p)=1.\]

For the identity element,
\[c_{C_p}(e)
=
1+\frac1p.\]
For every non-identity element $g\in C_p$,
\[\langle g\rangle=C_p,\]
and hence
\[c_{C_p}(g)=\frac1p.\]

Thus
\[b_{C_p}(1)=1+\frac1p,
\qquad
b_{C_p}(C_p)=\frac1p.\]
M\"obius inversion gives
\[
\begin{aligned}
\frac1{|1|}
&=
\mu(1,1)b_{C_p}(1)
+
\mu(1,C_p)b_{C_p}(C_p)\\
&=
\left(1+\frac1p\right)-\frac1p\\
&=
1,
\end{aligned}
\]
while
\[\frac1{|C_p|}
=
\mu(C_p,C_p)b_{C_p}(C_p)
=
\frac1p.\]

\subsubsection{Example: The symmetric group $S_3$}

Let $G=S_3$. Its subgroup lattice consists of the trivial subgroup,
three subgroups of order $2$, one subgroup of order $3$, and the whole
group $S_3$.

Let
\[T=\langle(1\,2)\rangle,
\qquad
C=\langle(1\,2\,3)\rangle.\]
Then
\[|T|=2,
\qquad
|C|=3.\]

The upper zeta transform takes the values
\[b_{S_3}(1)
=
1+3\cdot\frac12+\frac13+\frac16
=
3,\]
\[b_{S_3}(T)
=
\frac12+\frac16
=
\frac23,\]
\[b_{S_3}(C)
=
\frac13+\frac16
=
\frac12,\]
and
\[b_{S_3}(S_3)
=
\frac16.\]

Consequently,
\[c_{S_3}(e)=3,\]
\[c_{S_3}(t)=\frac23
\qquad
\text{for every transposition }t,\]
and
\[c_{S_3}(c)=\frac12
\qquad
\text{for every $3$-cycle }c.\]

For the interval $[T,S_3]$ there are no subgroups strictly between $T$
and $S_3$. Hence
\[\mu(T,T)=1,
\qquad
\mu(T,S_3)=-1,\]
and therefore
\[
\begin{aligned}
\frac1{|T|}
&=
\mu(T,T)b_{S_3}(T)
+
\mu(T,S_3)b_{S_3}(S_3)\\
&=
\frac23-\frac16\\
&=
\frac12.
\end{aligned}
\]

Similarly,
\[\mu(C,C)=1,
\qquad
\mu(C,S_3)=-1,\]
and hence
\[
\begin{aligned}
\frac1{|C|}
&=
\mu(C,C)b_{S_3}(C)
+
\mu(C,S_3)b_{S_3}(S_3)\\
&=
\frac12-\frac16\\
&=
\frac13.
\end{aligned}
\]

Finally,
\[\mu(1,1)=1,
\qquad
\mu(1,T_i)=-1,
\qquad
\mu(1,C)=-1,
\qquad
\mu(1,S_3)=3,\]
where \(T_1,T_2,T_3\) are the three subgroups of order \(2\). Thus
\[
\begin{aligned}
\frac1{|1|}
&=
b_{S_3}(1)
-
\sum_{i=1}^{3}b_{S_3}(T_i)
-
b_{S_3}(C)
+
3b_{S_3}(S_3)\\
&=
3-3\cdot\frac23-\frac12+3\cdot\frac16\\
&=
1.
\end{aligned}
\]

\subsection{A probabilistic interpretation of the coefficients}
\label{subsec:probabilistic-interpretation}

The coefficients $c_G(g)$ also admit a natural probabilistic
interpretation. After normalization, they describe the distribution
obtained by first choosing a subgroup uniformly at random and then
choosing an element uniformly from that subgroup.

\begin{proposition}\label{prop:sum-cG}
Let $G$ be a finite group. Then
\[\sum_{g\in G}c_G(g)
=
|\mathcal{H}(G)|,\]
where $\mathcal{H}(G)$ denotes the set of all subgroups of $G$.
\end{proposition}

\begin{proof}
By definition,
\[c_G(g)
=
\sum_{\substack{H\leq G\\g\in H}}
\frac1{|H|}.\]
Therefore
\[
\begin{aligned}
\sum_{g\in G}c_G(g)
&=
\sum_{g\in G}
\sum_{\substack{H\leq G\\g\in H}}
\frac1{|H|}\\
&=
\sum_{H\leq G}
\frac1{|H|}
\sum_{g\in H}1\\
&=
\sum_{H\leq G}1\\
&=
|\mathcal{H}(G)|.
\end{aligned}
\]
\end{proof}

It follows that
\[\widehat{c}_G(g)
:=
\frac{c_G(g)}{|\mathcal{H}(G)|}\]
defines a probability distribution on $G$.

\begin{theorem}[Probabilistic interpretation]
\label{thm:probabilistic}
Let $G$ be a finite group. Consider the following random experiment:
\begin{enumerate}
    \item choose a subgroup $H\leq G$ uniformly at random from
    $\mathcal{H}(G)$;
    \item conditional on $H$, choose an element $g\in H$ uniformly at
    random.
\end{enumerate}
Then the resulting probability distribution on $G$ is
\[\mathbb{P}(g)
=
\widehat{c}_G(g)
=
\frac{c_G(g)}{|\mathcal{H}(G)|}.\]

Consequently, if
\[\widehat{F}_{\Gamma}(G;x,y)
:=
\frac{1}{|\mathcal{H}(G)|}
F_{\Gamma}(G;x,y),\]
then
\[\widehat{F}_{\Gamma}(G;x,y)
=
\sum_{g\in G}
\widehat{c}_G(g)P_{\Gamma/g}(x,y)
=
\mathbb{E}\!\left[P_{\Gamma/g}(x,y)\right].\]
\end{theorem}

\begin{proof}
For a fixed $g\in G$, the probability that the experiment produces $g$
is
\[
\begin{aligned}
\mathbb{P}(g)
&=
\sum_{\substack{H\leq G\\g\in H}}
\mathbb{P}(H)\mathbb{P}(g\mid H)\\
&=
\sum_{\substack{H\leq G\\g\in H}}
\frac{1}{|\mathcal{H}(G)|}\frac1{|H|}\\
&=
\frac{c_G(g)}{|\mathcal{H}(G)|}\\
&=
\widehat{c}_G(g).
\end{aligned}
\]

Using the element-wise expansion,
\[F_{\Gamma}(G;x,y)
=
\sum_{g\in G}
c_G(g)P_{\Gamma/g}(x,y),\]
and dividing by $|\mathcal{H}(G)|$ gives
\[\widehat{F}_{\Gamma}(G;x,y)
=
\sum_{g\in G}
\widehat{c}_G(g)P_{\Gamma/g}(x,y),\]
which is the corresponding expectation.
\end{proof}

\begin{remark}
The distribution $\widehat{c}_G$ is generally not uniform on $G$.
Elements contained in many subgroups receive larger weights. In
particular,
\[\widehat{c}_G(1)
=
\frac{1}{|\mathcal{H}(G)|}
\sum_{H\leq G}\frac1{|H|}\]
is maximal among the probabilities $\widehat{c}_G(g)$.
\end{remark}

The normalization therefore gives the multiorbital polynomial the
interpretation
\[F_{\Gamma}(G;x,y)
=
|\mathcal{H}(G)|
\,
\mathbb{E}\!\left[P_{\Gamma/g}(x,y)\right],\]
where $g$ is generated by the two-stage random experiment above.

\subsubsection{Example: The induced distribution for $S_3$}

The group $S_3$ has six subgroups: the trivial subgroup, three subgroups
of order $2$, one subgroup of order $3$, and $S_3$ itself. Thus
\[|\mathcal{H}(S_3)|=6.\]

The coefficient values are
\[c_{S_3}(e)=3,
\qquad
c_{S_3}(t)=\frac23\]
for every transposition $t$, and
\[c_{S_3}(c)=\frac12\]
for every $3$-cycle $c$.

Consequently,
\[\widehat c_{S_3}(e)
=
\frac12,\]
\[\widehat c_{S_3}(t)
=
\frac19
\qquad
\text{for every transposition }t,\]
and
\[\widehat c_{S_3}(c)
=
\frac1{12}
\qquad
\text{for every $3$-cycle }c.\]

Indeed,
\[\frac12
+
3\cdot\frac19
+
2\cdot\frac1{12}
=
1.\]

Thus, for a graph $\Gamma$ equipped with an $S_3$-action,
\[\widehat F_{\Gamma}(S_3;x,y)
=
\frac12P_{\Gamma/e}(x,y)
+
\frac13P_{\Gamma/t}(x,y)
+
\frac16P_{\Gamma/c}(x,y),\]
where $t$ denotes any transposition and $c$ any $3$-cycle.

The factors $\frac13$ and $\frac16$ arise because there are three
transpositions and two $3$-cycles, respectively. Equivalently,
\[\widehat F_{\Gamma}(S_3;x,y)
=
\mathbb{E}\!\left[P_{\Gamma/g}(x,y)\right]\]
for the probability distribution above.

\subsection{Connections to classical invariants}

The multiorbital polynomial admits several natural specializations.

\begin{proposition}[Chromatic specialization]
By the specialization $P_{\Lambda}(x,x)=\chi_{\Lambda}(x)$ of
Section~\ref{subsec:bivariate}, the multiorbital polynomial satisfies at
$y=x$
\[F_{\Gamma}(G;x,x)
=
\sum_{H\leq G}
\frac1{|H|}
\sum_{h\in H}
\chi_{\Gamma/h}(x).\]
Equivalently,
\[F_{\Gamma}(G;x,x)
=
\sum_{g\in G}
c_G(g)\chi_{\Gamma/g}(x).\]
\end{proposition}

\begin{proof}
Substituting $y=x$ into
\eqref{eq:multi-orbital-expanded} gives
\[F_{\Gamma}(G;x,x)
=
\sum_{H\leq G}
\frac1{|H|}
\sum_{h\in H}
P_{\Gamma/h}(x,x).\]
The specialization $P_{\Gamma/h}(x,x)=\chi_{\Gamma/h}(x)$ of
Section~\ref{subsec:bivariate} yields
\[P_{\Gamma/h}(x,x)
=
\chi_{\Gamma/h}(x),\]
which proves the first identity. The second follows from
\eqref{eq:coefficient-expansion}.
\end{proof}

\begin{remark}
The chromatic specialization retains the averaging over the individual elements of each subgroup. In particular, the subgroup-indexed polynomial \(OP_{\Gamma,H}\) specializes to the average of the chromatic polynomials \(\chi_{\Gamma/h}(x)\) over \(h\in H\), rather than to their unweighted sum.
\end{remark}

\begin{proposition}[Edgeless graphs and a weighted cycle index]
Let $\Gamma$ be an edgeless graph on $n$ vertices, and suppose that $G$
acts on $V(\Gamma)$. Then, for every $g\in G$, the quotient $\Gamma/g$
is edgeless and has
\[\#\operatorname{Orb}_{\langle g\rangle}(V(\Gamma))\]
vertices. Hence
\[P_{\Gamma/g}(x,y)
=
x^{\,\#\operatorname{Orb}_{\langle g\rangle}(V(\Gamma))}.\]
Consequently,
\begin{equation}\label{eq:weighted-cycle-index}
F_{\Gamma}(G;x,y)
=
\sum_{g\in G}
c_G(g)\,
x^{\,\#\operatorname{Orb}_{\langle g\rangle}(V(\Gamma))}.
\end{equation}
\end{proposition}

\begin{proof}
Since $\Gamma$ is edgeless, every quotient $\Gamma/g$ is also edgeless.
Its vertices are precisely the $\langle g\rangle$-orbits on
$V(\Gamma)$. An edgeless graph on $m$ vertices has bivariate chromatic
polynomial $x^m$. Therefore
\[P_{\Gamma/g}(x,y)
=
x^{\,\#\operatorname{Orb}_{\langle g\rangle}(V(\Gamma))}.\]
Substitution into \eqref{eq:coefficient-expansion} gives
\eqref{eq:weighted-cycle-index}.
\end{proof}

\begin{remark}
If $G$ is a permutation group acting on $n$ vertices, then
\[\#\operatorname{Orb}_{\langle g\rangle}(V(\Gamma))\]
is the number of cycles of the permutation $g$. Thus
\eqref{eq:weighted-cycle-index} becomes
\[F_{\Gamma}(G;x,y)
=
\sum_{g\in G}
c_G(g)x^{\#\operatorname{cycles}(g)}.\]
This has the form of a weighted cycle-index expression. It should not be
identified with the classical cycle index, since the weights $c_G(g)$ are
determined by the subgroup lattice rather than by the uniform factor
$1/|G|$.
\end{remark}


\section*{Conclusion}

We introduced the multiorbital bivariate chromatic polynomial
\[
F_{\Gamma}(G;x,y)
=
\sum_{H\leq G}
\frac{1}{|H|}
\sum_{h\in H}
P_{\Gamma/h}(x,y),
\]
which aggregates the orbital bivariate chromatic polynomials over the
subgroup lattice of the acting group.

The central structural identity is the element-wise expansion
\[
F_{\Gamma}(G;x,y)
=
\sum_{g\in G}
c_G(g)P_{\Gamma/g}(x,y),
\]
where
\[
c_G(g)
=
\sum_{\langle g\rangle\leq H\leq G}
\frac1{|H|}.
\]
The coefficient depends only on the cyclic subgroup $\langle g\rangle$,
which leads to a decomposition by cyclic subgroups. Since the coefficient
function is also constant on conjugacy classes, the polynomial admits a
corresponding conjugacy-class decomposition.

The subgroup-lattice structure of the coefficients is captured naturally
by M\"obius inversion. In addition, normalization by the number of
subgroups gives a probability distribution on $G$. Under this
distribution, the normalized multiorbital polynomial is the expected
quotient polynomial associated with the resulting random group element.

We further established a diagonal multiplicativity relation for disjoint
unions preserved by the group action and obtained a weighted cycle-index
expression for edgeless graphs. These properties show how the construction
combines subgroup-lattice data with the cyclic, conjugacy, and quotient
graph structures arising from the group action.

Finally, we answered a question posed by Dohmen and Lange-Geisler concerning
the distinguishing power of the orbital bivariate chromatic polynomial.
We exhibited two non-isomorphic graphs ($\Gamma_1 = P_3$ and $\Gamma_2 = K_2 \mathbin{\dot\cup} K_1$) with actions of the same group $C_2$ that yield
identical orbital bivariate chromatic polynomials but different
multiorbital bivariate chromatic polynomials. More precisely,
\[
OP_{\Gamma_1,C_2}(x,y)
=
OP_{\Gamma_2,C_2}(x,y),
\]
while
\[
F_{\Gamma_1}(C_2;x,y)
\neq
F_{\Gamma_2}(C_2;x,y).
\]
Thus, the orbital bivariate chromatic polynomial does not distinguish
non-isomorphic graphs in general. Moreover, the multiorbital
bivariate chromatic polynomial is not determined by the orbital bivariate
chromatic polynomial: the subgroup-lattice aggregation retains information
that can distinguish graph actions with the same orbital polynomial.

This example establishes a genuine separation between the two invariants
and demonstrates that the subgroup-lattice aggregation is not merely a
reformulation of the orbital bivariate chromatic polynomial.

A natural question arising from the present results concerns the
distinguishing power of the multiorbital bivariate chromatic polynomial
itself. Although the example above shows that the multiorbital polynomial
can distinguish graph actions that have the same orbital bivariate
chromatic polynomial, it remains open whether the multiorbital
polynomial is a complete invariant for graph actions. In particular, one
may ask whether there exist non-isomorphic graph actions $(\Gamma_1,G)$ and
$(\Gamma_2,G)$ such that
\[
F_{\Gamma_1}(G;x,y)
=
F_{\Gamma_2}(G;x,y).
\]
Furthermore, the converse relation—whether equality of multiorbital polynomials $F_{\Gamma_1}(G) = F_{\Gamma_2}(G)$ implies equality of their orbital counterparts $OP_{\Gamma_1,G} = OP_{\Gamma_2,G}$—remains an open problem for future research.

\section*{Declaration of Generative AI and AI-assisted Technologies}

During the development and preparation of this manuscript, I used
ChatGPT and other generative AI tools as mathematical discussion
partners and for drafting the manuscript.

The central research question, the construction of the proposed
polynomial, its defining coefficient function, and the idea of
diagonal multiplicativity originated with me. During the subsequent
exploration of mathematical properties, AI tools suggested Möbius
inversion as a possible approach and contributed to the development
of a probabilistic interpretation of the coefficients.

AI tools also assisted in exploring and formulating results concerning
the highest-degree coefficients of the orbital bivariate chromatic
polynomial. This included the use of the broken-circuit expansion to
derive and present combinatorial interpretations of selected
coefficients in terms of the acting group, its action on the vertex
set, and the edges of the graph.

All illustrative examples in the manuscript were generated with AI
assistance. This includes the example demonstrating that two
non-isomorphic graphs can have identical orbital bivariate chromatic
polynomials but different multiorbital bivariate chromatic
polynomials. I independently recalculated this distinguishing example
and verified the stated polynomial identities.

AI tools also assisted in drafting the mathematical exposition,
including definitions, propositions, proofs, and examples, and in
preparing the manuscript in LaTeX. The author is responsible for
the final manuscript and its mathematical content.



\begin{thebibliography}{99}

\bibitem{AverbouchGodlinMakowsky2010}
I. Averbouch, B. Godlin, and J. A. Makowsky,
\emph{An extension of the bivariate chromatic polynomial},
European Journal of Combinatorics, 31(1) (2010), 1--17.
doi:10.1016/j.ejc.2009.05.006.

\bibitem{Birkhoff1912}
G. D. Birkhoff,
\emph{A Determinant Formula for the Number of Ways of Coloring a Map},
Annals of Mathematics, 14(1) (1912), 42--46.

\bibitem{CameronKayibi2007}
P. J. Cameron and K. K. Kayibi,
\emph{Orbital Chromatic and Flow Roots},
Combinatorics, Probability and Computing, 16(3) (2007), 401--407.
doi:10.1017/S0963548306008200.

\bibitem{CameronJacksonRudd2008}
P. J. Cameron, B. Jackson, and J. D. Rudd,
\emph{Orbit-counting polynomials for graphs and codes},
Discrete Mathematics, 308(5--6) (2008), 920--930.
doi:10.1016/j.disc.2007.07.108.

\bibitem{DohmenLangeGeisler2026}
K. Dohmen and M. Lange-Geisler,
\emph{The Orbital Bivariate Chromatic Polynomial},
arXiv:2009.08235 [math.CO], version 6, 2026.
\url{https://arxiv.org/abs/2009.08235v6}

\bibitem{DohmenPoenitzTittmann2003}
K. Dohmen, A. P\"onitz, and P. Tittmann,
\emph{A new two-variable generalization of the chromatic polynomial},
Discrete Mathematics and Theoretical Computer Science, 6(1) (2003),
69--90.
doi:10.46298/dmtcs.335.

\bibitem{Klass1976}
M. J. Klass,
\emph{A generalization of Burnside's combinatorial lemma},
Journal of Combinatorial Theory, Series A, 20(3) (1976), 273--278.
doi:10.1016/0097-3165(76)90021-2.

\end{thebibliography}
\end{document}